\documentclass{amsart}
\usepackage{amsmath}
\usepackage{amssymb}
\usepackage[all]{xy}
\usepackage{url}

\newcommand{\fh}{\mathfrak h}

\newcommand{\cA}{\mathcal A}

\newcommand{\cF}{\mathcal F}
\newcommand{\cM}{\mathcal M}
\newcommand{\cX}{\mathcal X}

\renewcommand{\AA}{\mathbb A}
\newcommand{\BB}{\mathbb B}
\newcommand{\CC}{\mathbb C}
\newcommand{\NN}{\mathbb N}
\newcommand{\QQ}{\mathbb Q}
\newcommand{\RR}{\mathbb R}
\newcommand{\UU}{\mathbb U}
\newcommand{\ZZ}{\mathbb Z}

\newcommand{\Ga}{{\mathbb G}_a}
\newcommand{\Gm}{{\mathbb G}_m}

\theoremstyle{theorem}
\newtheorem{theorem}{Theorem}[section]
\newtheorem{lem}[theorem]{Lemma}
\newtheorem{cor}[theorem]{Corollary}
\newtheorem{prop}[theorem]{Proposition}
\newtheorem{question}[theorem]{Question}

\newtheorem{problem}[theorem]{Problem}

\theoremstyle{definition}
\newtheorem{Def}[theorem]{Definition}
\newtheorem{nota}[theorem]{Notation}
\newtheorem{propdef}[theorem]{Proposition/Definition}

\theoremstyle{remark}
\newtheorem{Rk}[theorem]{Remark}

\begin{document}

\title{Algebraic differential equations from covering maps}
\author{Thomas Scanlon} \thanks{Partially supported by NSF grant DMS-1001556.  This paper was completed while
the author was in residence at MSRI during the Spring 2014 semester for the Model Theory and Number Theory program
which was partially supported by NSF grant 0932078 000.  The author thanks D. Bertrand, A. Buium, J. Freitag, B. Mazur, Y. Peterzil,
M. Singer, and S. Starchenko for their very helpful discussions about this work.}
\email{scanlon@math.berkeley.edu}
\address{University of California, Berkeley \\
Department of Mathematics \\
Evans Hall \\
Berkeley, CA 94720-3840 \\
USA}
\begin{abstract}
Let $Y$ be a complex algebraic variety, $G \curvearrowright Y$ an
action of an algebraic group on $Y$, $U \subseteq Y(\CC)$ a
complex submanifold, $\Gamma < G(\CC)$ a discrete, Zariski dense
subgroup of $G(\CC)$ which preserves $U$, and $\pi:U \to X(\CC)$
an analytic covering map of the complex algebraic variety $X$ expressing
$X(\CC)$ as $\Gamma \backslash U$.  We note that the theory of
elimination of imaginaries in differentially closed fields
produces a generalized Schwarzian derivative
$\widetilde{\chi}:Y \to Z$ (where $Z$ is some algebraic variety)
expressing the quotient of $Y$ by the action of the constant
points of $G$.  Under the additional hypothesis that
the restriction of $\pi$ to some set containing a fundamental
domain is definable in an o-minimal expansion of the real field,
we show as a consequence of the Peterzil-Starchenko o-minimal
GAGA theorem that the \emph{prima facie} differentially analytic
relation $\chi := \widetilde{\chi} \circ \pi^{-1}$ is
a well-defined, differential constructible function.
The function $\chi$
nearly inverts $\pi$ in the sense that for any differential
field $K$ of meromorphic functions, if $a, b \in X(K)$ then
$\chi(a) = \chi(b)$ if and only if after suitable restriction there
is some $\gamma \in G(\CC)$ with $\pi(\gamma \cdot \pi^{-1}(a)) = b$.
\end{abstract}

\maketitle

\section{Introduction}

As is well-known, the complex exponential function $\exp:\CC \to \CC^\times$
admits a local analytic inverse, the logarithm function, but the logarithm
cannot be made into a globally defined analytic function. The ambiguity
in the choice of a branch of the logarithm comes from addition by an element of
the discrete group $2 \pi i \ZZ$.  Hence, if we regard the logarithm as
acting on meromorphic function via $f \mapsto \log \circ f$, then while the operator
$f \mapsto \log(f)$ is not well-defined, the logarithmic derivative,
$f \mapsto \frac{d}{dz} (\log(f))$ is.   Of course, more is true in that the
logarithmic derivative is given by the simple differential algebraic
formula $\frac{d}{dz}(\log(f)) = \frac{f'}{f}$.

That the logarithmic derivative is a differential rational function
admits various proofs ranging from a direct computation, to Kolchin's
general theory of logarithmic differentiation on algebraic groups~\cite{Kolchin-DAAG}, to the
techniques we employ in this paper.   In Section~\ref{examplessect} we discuss the algebraic
construction of the logarithmic derivative in detail.

The purpose of this paper is to show that under very general hypotheses, differential
analytic operators constructed by inverting analytic covering maps and then applying
differential operators to kill the action of the constant points of some algebraic group
are actually differential algebraic.  Let us describe more precisely what we have in mind.
We are given an algebraic group $G$ over $\CC$, a complex algebraic variety $Y$, a
regular action of $G$ on $Y$, a complex submanifold $U \subseteq Y(\CC)$ of $Y(\CC)$,
a Zariski dense subgroup $\Gamma < G(\CC)$ for which $\Gamma$ preserves $U$ and an analytic covering
map $\pi:U \to X(\CC)$ expressing the complex algebraic variety $X$ as the quotient
$\Gamma \backslash U$.  Because $\pi$ is a covering map, the inverse
$\pi^{-1}:X(\CC) \to U$ is a many-valued analytic function, well-defined up to the action
of $\Gamma$.  Using the theory of elimination of imaginaries in differential
fields, we produce a differential constructible function, which we call a generalized logarithmic
derivative associated to $\pi$, $\widetilde{\chi}:Y \to Z$
(where $Z$ is an algebraic variety) so that for any differential field $M$ having field of
constants $\CC$ and points $a, b \in Y(M)$ one has $\widetilde{\chi}(a) = \widetilde{\chi}(b)$ if and only if
there is some some $\gamma \in G(\CC)$ with $\gamma \cdot a = b$.  Hence, the
differential analytic operator $\chi := \widetilde{\chi} \circ \pi^{-1}$ gives
a well defined function $X(M) \to Z(M)$ for $M$ any field of meromorphic functions.

Under a mild hypothesis on $\pi$, namely that there is an o-minimal expansion
of $\RR$ in which there is a definable subset $F \subseteq Y(\CC)$ for which
the restriction of $\pi$ is definable and surjective onto $X(\CC)$, we then deduce from
a remarkable theorem of Peterzil-Starchenko that the \emph{a priori} differential
\emph{analytically} constructible function $\chi$ is in fact differential \emph{algebraically}
constructible.    This definability hypothesis holds in many cases of interest, such as for the
covering maps associated to moduli spaces of abelian varieties and of the universal families
of abelian varieties over these moduli spaces.

Our work on this problem was motivated by our attempt to understand Buium's construction
of differential rational functions on moduli spaces of abelian varieties whose fibres
are finite dimensional differential varieties containing the isogeny classes encoded by such moduli points~\cite{Buium}.
His construction is algebraic in the style of Kolchin's construction of logarithmic derivations,
though much more sophisticated.   Buium's maps are differential rational, meaning that they have
a non-trivial indeterminacy locus; ours are defined everywhere, though they are intrinsically
differential constructible, meaning that they are only piecewise given by differential regular
functions.  We seek an algebraic interpretation, for example in terms of a variant of the notion
of $\delta$-Hodge structure, for our maps on Buium's indeterminacy locus.  More generally,
knowing that the map $\chi$ is differential algebraic, one expects a direct algebraic
construction.  We speculate about this in Section~\ref{cq}.

This paper is organized as follows.  We begin in Section~\ref{prelim}
by recalling some of the basics of differential algebra, complex analysis
and especially the Peterzil-Starchenko theory of o-minimal complex analysis.
In Section~\ref{mt} we state precisely and prove our main theorem.  In Section~\ref{theorygld} we develop some of the
basic theory of generalized logarithmic derivatives.  In Section~\ref{examplessect}
we discuss specific examples of covering maps to which our main theorem applies.  In particular,
we note the existence of differential constructible functions whose fibres are the Kolchin closures of
isogeny classes in moduli spaces of abelian varieties.
We also discuss the problem of extending our main theorem to the context of Picard-Fuchs equations 
associated to families of varieties and to the analytic construction of Manin homomorphisms for nonconstant
abelian varieties coming from families of
covering maps. We close in Section~\ref{cq} with some natural questions.

\section{Preliminaries}
\label{prelim}

As our main theorem involves a comparison of three different kinds of structures,
namely, complex manifolds, differential algebraic functions, and o-minimally definable sets,
it should come as no surprise that sometimes the same object may be considered differently
in each of these domains.   In this section we establish our notation and explain how these
theories interact with each other.

\subsection{Jet spaces}

We shall use a construction of higher order tangent spaces which goes variously under the
names of spaces of jets, arc spaces, and prolongation spaces.  We begin by recalling the
differential geometric jet spaces modified slightly for the complex analytic category.
The reader can find details of the space of jets construction in Chapter 12 of~\cite{Bourbaki},
though the discussion there covers only jets of maps from one manifold to another while
we shall consider jets of germs.    The extension to our case is routine.

We begin with our notation for balls and polydiscs.

\begin{nota}
We denote the unit disc in the complex plane by
$$\BB := \{z \in \CC ~:~ |z| < 1 \} \text{ .}$$  For a natural number $n$, the
polydisc
$$\BB^n := \{ (z_1,\ldots,z_n) \in \CC^n ~:~ |z_i| < 1 \text{ for } i \leq n \}$$
is the $n^\text{th}$ Cartesian power of $\BB$.  More generally, for $0 < r < 1$
we denote the disc of radius $r$ centered at the origin by $\BB_r$ and its $n^\text{th}$
Cartesian power by $\BB_r^n$.
\end{nota}

Let us recall the construction of the space of jets.

\begin{Def}
 If $M$ is a complex manifold and
$f:\BB_r^n \to M$ and $g:\BB_s^n \to M$ are two analytic maps into $M$ from polydiscs
of some radii $r$ and $s$ and $m$ is a natural number, then we say that $f$ and $g$ have the
same jets up to order $m$ if there is a coordinate neighborhood $U \ni f(0) = g(0)$ with
$\eta:U \cong \BB^{\dim(M)}$ so that for each $j \leq \dim(M)$ and each
multi-index $\alpha = (\alpha_1,\ldots,\alpha_n)$ with $|\alpha| = \sum \alpha_i \leq m$
one has $\frac{{\partial}^{|\alpha|}}{\partial z_1^{\alpha_1} \cdots \partial z_n^{\alpha_n}} (\eta_j \circ f) (0) =
 \frac{\partial^{|\alpha|}}{\partial z_1^{\alpha_1} \cdots \partial z_n^{\alpha_n}} (\eta_j \circ g) (0)$.
We write $[f]_m$ for the equivalence class of $f$ with respect to the equivalence relation of having
the same jets up to order $m$.  Note that $[f]_m$ depends only on the germ of $f$ at the origin.
\end{Def}

The set of equivalence classes of maps from $n$-dimensional polydiscs into $M$ may be given the structure
of a complex manifold, $J_{m,n}(M)$, and comes equipped with an analytic map $J_{m,n}(M) \to M$ making
$J_{m,n}(M)$ into an affine bundle over $M$.
The construction is functorial in that if $\phi:M \to N$ is
 a map of complex manifolds, then there is an induced morphism $J_{m,n}(\phi):J_{m,n}(M) \to J_{m,n}(N)$
 given by $J_{m,n}([f]_m) := [\phi \circ f]_{m,n}$.

The space of jets construction has an algebraic geometric counterpart with the notion of arc spaces.
For an introduction to arc spaces see~\cite{DenefLoeser} in which what we call \emph{arc spaces} are called
\emph{jet spaces}.  For details of these constructions at the level of generality
we use in this paper, see~\cite{MoosaScanlon}.

\begin{propdef}
If $X$ is a scheme over $\CC$, then for each $n$ and $m$ the functor from ${\mathbb C}$-algebras to sets given
by $R \mapsto X(R[\epsilon_1,\ldots,\epsilon_n]/(\epsilon_1,\ldots,\epsilon_n)^{m+1})$
is represented by a scheme $\cA_{m,n}(X)$.  In the literature, this construction is usually limited to the
case of $n = 1$ for which $\cA_{m,1}(X) =: \cA_m(X)$ is the \emph{$m^\text{th}$ arc bundle} of $X$.
\end{propdef}

\begin{Rk}
We shall use the more general $n$ to discuss partial differential equations.   To ease readability, when no confusion
would arise we shall
suppress the subscript $n$.
\end{Rk}

In general, even if $X$ is an algebraic variety, it may happen
that $\cA_{m}(X)$ is non-reduced.  However, when $X$ is a smooth variety, one has $\cA_{m} (X)(\CC) = J_{m} ( X(\CC))$
and we will need the arc bundle construction only in the case of smooth varieties.

\subsection{Differential algebra}

Let us recall some of the basics of differential algebra for which the book~\cite{Kolchin-DAAG} is the
standard reference.  For an introduction to the model theory of
differential fields, see~\cite{MMP}.

\begin{Def}
A \emph{differential ring} $(K,\Delta)$ is a (commutative) ring $K$ given together with a finite
sequence $\Delta = \langle \partial_1, \ldots, \partial_n \rangle$ of commuting derivations.  If $K$ is a
field, then we call $(K,\Delta)$ a \emph{differential field}. Generally, we write $K$ for the tuple $(K,\Delta)$.
When $|\Delta| = 1$, we call $K$ an \emph{ordinary differential field} and otherwise $K$ is a \emph{partial
differential field}.  In what follows we shall work only with differential fields of characteristic zero
and the phrase \emph{differential field} shall mean \emph{differential field of characteristic zero}.
A map of differential rings $f:(A,\partial_1,\ldots,\partial_n) \to (B,\partial_1,\ldots,\partial_n)$
is given by a map of rings $f:A \to B$ which respects the derivations in the sense that $f \circ \partial_i =
\partial_i \circ f$ for all $i \leq n$.   By a \emph{$K$-$\Delta$-algebra} we mean a differntial
ring $(A,\Delta)$ given together with a $K$-algebra structure for which the map $K \to A$ is a
map of differential rings.
\end{Def}

A standard example of a differential field is given by taking a connected open
domain $U \subseteq \CC^n$ and setting $(K,\partial_1,\ldots,\partial_n) = ({\mathcal M}(U),\frac{\partial}{\partial z_1},
\ldots, \frac{\partial}{\partial z_n})$ where ${\mathcal M}(U)$ the field of meromorphic functions on $U$ where
$z_1,\ldots,z_n$ are the standard coordinates on $U$.  The
Seidenberg embedding theorem asserts that every finitely generated differential field of characteristic zero may
be realized as a subdifferential field of ${\mathcal M}(U)$ for some connected domain $U \subseteq \CC^n$.  In fact, a little
more is true.

\begin{theorem}[Seidenberg~\cite{Seidenberg1,Seidenberg2}]
\label{seidenbergembedding}
If $K \subseteq {\mathcal M}(U)$ is a finitely generated differential subfield of
the differential field of meromorphic functions on some connected domain $U \subseteq \CC^n$
and $L = K \langle u \rangle$ is a simple differential extension of $K$ (meaning that $L$ is a
differential field extending $K$ which is generated as a differential field by $K$ and the single element $u \in L$),
then there is a connected domain $V \subseteq U$ and an embedding of differential fields
$\iota:L \hookrightarrow  {\mathcal M}(V)$ compatible with the embedding ${\mathcal M}(U) \to {\mathcal M}(V)$.
\end{theorem}

To speak of differential equations on an algebraic variety, we need to generalize the arc space construction.

Consider $(K,\partial_1,\ldots,\partial_n)$ a differential field.   For each $m \in \NN$ there are
two natural $K$-algebra structures on $R_m = K[\epsilon_1,\ldots,\epsilon_n]/(\epsilon_1,\ldots,\epsilon_n)^{m+1}$.
First, there is the standard structure $\iota:K \to R_m$ given by $a \mapsto a + \sum_{{\mathbf 0} \neq \alpha} 0 \cdot \epsilon^\alpha$.
Secondly, using the derivations, we have an exponential map $E_m:K \to R_m$ given by
$$a \mapsto \sum _{|\alpha| \leq m}  \frac{1}{\alpha_1! \cdots \alpha_n!} \partial_1^{\alpha_1} \cdots \partial_n^{\alpha_n} (x) \epsilon_1^{\alpha_1} \cdots \epsilon_n^{\alpha_n} \text{ .}$$
Note that the standard map $\iota$ is itself an exponential map with respect to the trivial derivations.

If $X$ is an algebraic variety over $K$, then the functor from $K$-algebras to $\mathsf{S}\mathsf{e}\mathsf{t}$ defined by
$$A \mapsto (X \times_{E_m^*} \operatorname{Spec}(R_m))(A \otimes_{\iota} R_m)$$
is representable by a scheme $\tau_m X$ called the \emph{$m^\text{th}$ prolongation space of $X$}.  Note that when
the derivations are trivial, then $E_m = \iota$ and $\tau_m X = \cA_{m,n} X$.

From the description of $\tau_m X$ as a functor, one sees that there are natural projection maps $\pi_{m,m'}:\tau_m X \to \tau_{m'} X$
for $m \geq m'$ corresponding to the reduction maps $R_{m} \to R_{m'}$ and that $\tau_0 X = X$ canonically.  On the other hand,
corresponding to the exponential map $E_m:K \to R_m$ one has a differentially defined map $\nabla_m:X \to \tau_m X$.  That is,
for any $K$-$\Delta$-algebra $A$ we have a map of sets $\nabla_m:X(A) \to \tau_m X(A)$ corresponding to the map
$$X(A) \to (X \times_{E_m^*} \operatorname{Spec}(R_n)) (A \otimes_\iota R_m)$$ coming from $1_A \otimes_K E_m$.

Using the prolongation spaces, one can make sense of the notions of differential regular (respectively, rational or
constructible) functions on the algebraic variety $X$.  That is, a differential regular (respectively, rational or constructible)
function $f:X \to Y$ (where $Y$ is another algebraic variety over $K$) is a map from $X$ to $Y$, considered
as functors from the category of $K$-$\Delta$-algebras to $\mathsf{S}\mathsf{e}\mathsf{t}$ which
is given by a regular (respectively, rational
or constructible) function $f_m:\tau_m X \to Y$ for some $m \in \NN$ in the sense that for any
$K$-$\Delta$-algebra $A$ and point $a \in X(A)$ one has $f(a) = f_m(\nabla_m(a))$.

A differential subvariety $Y$  of $X$ is given by a subvariety $Y_m \subseteq  \tau_m X$ for some $m \in \NN$.
For any $K$-$\Delta$-algebra $A$, the set of $A$-points on $Y$ is
$$Y(A) := \{ a \in X(A) ~:~ \nabla_m(a) \in Y_m(A) \} \text{ .}$$

Sets of the form $Y(K) \subseteq X(K)$ are called \emph{Kolchin closed} (or, sometimes, \emph{$\Delta$-closed}) and
as the name suggests are the closed sets of a noetherian topology on $X(K)$.  If $Y$ is an irreducible
differential subvariety of the algebraic variety $X$, then the set of differential rational functions on $Y$ forms a
differential field denoted by  $K \langle Y \rangle$.  If $\operatorname{tr.deg}_K(K \langle Y \rangle)$ is finite,
then we say that $Y$ is finite dimensional and define $\dim Y := \operatorname{tr.deg}_K(K \langle Y \rangle)$.

An important class of Kolchin closed sets comes from the constants.

\begin{Def}
If $(R,\Delta)$ is a differential ring, then
$$R^\Delta := \{ a \in R ~:~ \partial(a) = 0 \text{ for all } \partial \in \Delta \}$$
is the \emph{ring of constants}.  We sometimes write $C = C_R$ for $R^\Delta$.
\end{Def}

If $(K,\partial_1,\ldots,\partial_n)$ is a differential field and
 $X$ is an algebraic variety over the constants $K^\Delta$, then as we noted above $\tau_m (X_K) = \cA_{m,n} (X_K)$, canonically, and there
is a natural regular section $z_m:X_K \to \tau_m(X_K)$ of the projection map $\tau_m (X_K) \to \tau_0 (X_K) = X_K$ corresponding
to the standard map $\iota:K \to R_m$.  We define $X^\Delta$ to be the differential subvariety of $X$ given by $z(X_K) \subseteq \tau_1 (X_K)$.
At the level of points, if $A$ is $K$-$\Delta$-algebra, then $X^\Delta(A) = X(A^\Delta) \subseteq X(A)$.
Note that if $X$ is an algebraic variety over the constants, then its dimension as an algebraic variety is equal to the dimension of
$X^\Delta$ as a differential variety.

Just as algebraically closed fields play the role of universal domains for ordinary algebraic varieties, differentially closed
fields serve as universal domains for differential algebraic geometry.  Here a \emph{differentially closed field} $(\UU,\partial_1,\ldots,\partial_n)$
is an existentially closed differential field in the sense that if $X$ is an algebraic variety over $\UU$ and
$Y \subseteq X$ is a differential subvariety for which there is some differential field extension $(L,\partial_1,\ldots,\partial_n)
\supseteq (\UU,\partial_1,\ldots,\partial_n)$ with $Y(L) \neq \varnothing$, then $Y(\UU) \neq \varnothing$. The
theory of differentially closed fields (of characteristic zero in $n$ commuting derivations) is axiomatized by a
first-order theory $\operatorname{DCF}_{0,n}$ in the language of rings augmented by $n$ unary function symbols to be interpreted
as the distinguished derivations.   For us, the crucial facts about $\operatorname{DCF}_{0,n}$ are:
\begin{itemize}
\item $\operatorname{DCF}_{0,n}$ eliminates quantifiers.  Geometrically, this means that if $\UU$ is a differentially closed
field, $Y \subseteq X$ is a differential subvariety of an algebraic variety over $X$ and $f:Y \to Z$ is a differential rational
function, then $f(Y(\UU))$ is a differentially constructible subset of $Z(\UU)$. That is, it is a finite Boolean combination
of Kolchin closed sets.
\item It follows from Theorem~\ref{seidenbergembedding} that if $K \subseteq \cM(U)$ is a finitely generated
differential field of meromorphic functions on some domain $U$, then there is a differentially closed field $\UU$ containing $K$
which may be realized as a subfield of a differential field of germs of meromorphic functions at some point in $U$.
\item Finally, $\operatorname{DCF}_{0,n}$ \emph{eliminates imaginaries}.
\end{itemize}

The precise content of elimination of imaginaries is given by the following theorem of Poizat in the case of $n = 1$ and
of McGrail in general.

\begin{theorem}[Poizat~\cite{Poizat}, McGrail~\cite{McGrail}]
\label{eidcf}
The theory $\text{DCF}_{0,n}$ of differentially closed fields of characteristic zero with $n$ commuting derivations
eliminates imaginaries.  More concretely, if $\UU$ is a differentially closed field, $X \subseteq \UU^m$ is a
definable set and $E \subseteq X \times X$ is a definable equivalence relation, then there is a differential constructible
function $\eta:X \to \UU^\ell$ so that for $(a,b) \in X \times X$ one has $\eta(a) = \eta(b) \Longleftrightarrow a E b$.
Moreover, $\eta$ may be defined over the same parameters required for the definitions of $X$ and $E$.
\end{theorem}

\emph{Prima facie}, Theorem~\ref{eidcf} applies only to differentially closed fields, but the map $\eta$ may be taken to be defined
over the parameters required to define $X$ and $E$ and then from the constructibility of $\eta$ it is easy to see that for any differential
field $K$ over which everything is defined, if $a, b \in X(K)$, then $a E b \Longleftrightarrow \eta(a) = \eta(b)$.  The proof
of Theorem~\ref{eidcf} passes through the corresponding elimination of imaginaries theorem for algebraically closed fields (which is
also proven in detail in~\cite{Poizat}) and for our applications, we shall need to unwind this connection between elimination of
imaginaries for differentially closed fields and for ordinary algebraically closed fields.

\subsection{O-minimality}

Finally, we shall make use of the theory of o-minimality for which the book~\cite{vandenDries} is a good introduction.

\begin{Def}
An o-minimal structure on the real numbers $\RR_{\cF}$ is given by the choice of a distinguished set $\cF$ of
functions $f:\RR^n \to \RR$, where $n$ may depend on $f$, having the property that in the first-order structure
$$\RR_\cF := (\RR,+,\cdots,0,1,<,\{ f \}_{f \in \cF})$$ every definable subset of $\RR$ is a finite union of points and
intervals.
\end{Def}

\begin{Rk}
With the usual definition of an o-minimal structure, any expansion of the language of ordered sets is allowed.  However, it
follows from the existence of definable choice functions that every o-minimal expansion of an ordered field is bi-definable with
one obtained by expanding the language of ordered rings by function symbols.
\end{Rk}

That interesting o-minimal structures exist at all is highly non-trivial.  For us, the most important o-minimal
structure is $\RR_{\text{an},\exp}$ in which the set of distinguished functions consists of the real exponential function
$\exp:\RR \to \RR$ and local analytic functions.  That is, for each real analytic function $f$ defined on some open neighborhood of the
unit box $[-1,1]^n$ we are given the function $\overline{f}:\RR^n \to \RR$ defined by $\overline{f}(x_1,\ldots,x_n) = f(x_1,\ldots,x_n)$
if $(x_1,\ldots,x_n) \in [-1,1]^n$ and $\overline{f}(x_1,\ldots,x_n) = 0$ otherwise.   O-minimality of $\RR_{\text{an},\exp}$ is established
in~\cite{vdDM}.

O-minimality implies many strong regularity properties of the definable sets in any number of variables.
One result which we shall use is the existence of definable choice functions.  If $f:X \to Y$ is a definable (in the o-minimal
structure $\RR_\cF$) surjective function, then there is a definable right inverse $g:Y \to X$ (see Proposition 1.2 in~\cite{vandenDries}).

Since o-minimality is fundamentally a theory of ordered structures, it does not directly apply to complex analysis.  However, by realizing
$\CC$ as $\RR^2$ via the real and imaginary parts, one may interpret complex analysis within an o-minimal structure and in the series
of papers~\cite{PS1, PS2, PS3, PS-cag, PS4, PS5} Peterzil and Starchenko do just this.  For us, the most important result from their work is the following strengthening
of Chow's Theorem.

\begin{theorem}[Peterzil-Starchenko, Corollary 4.5 of~\cite{PS-cag}]
\label{PSGAGA}
Let $\RR_{\cF}$ be some o-minimal structure on the real numbers.
Let $X$ be a quasiprojective algebraic variety over $\CC$.  If $Y \subseteq X(\CC)$ is
an $\RR_{\cF}$-definable, closed complex analytic set in $X$, then $Y$ is
algebraic.
\end{theorem}

\begin{Rk}
The statement of Corollary 4.5 in~\cite{PS-cag} takes $X(\CC) = \CC^n$, but the proof uses only the fact that
$X$ embeds into a projective space.
\end{Rk}

\begin{Rk}
Note that in Theorem~\ref{PSGAGA} there is no hypothesis that  $X$ be projective nor that $f$ satisfy any kind of
growth condition towards the boundary of $X$ in some compactification.  On the other hand, one cannot completely avoid
such considerations in that to establish that the relevant analytic sets are definable it may be necessary to study their boundary behavior.
\end{Rk}

The proof of Theorem~\ref{PSGAGA} requires much of the Peterzil-Starchenko theory of o-minimal complex analysis, but the
basic idea is clear.  Many standard theorems in the theory of complex analysis assert that if some closed subset of a complex
manifold is generically analytic in some precise sense (for example, with respect to some reasonable dimension) and the ambient
space is sufficiently nice, then that set must be analytic
(see the theorems of Bishop~\cite{Bishop}, Remmert and Stein~\cite{RemmertStein}, and Shiffman~\cite{Shiffman}).
Sets definable in o-minimal structure enjoy a very smooth dimension
theory as do all of the sets naturally associated to them through standard geometric and elementary analytic constructions because these
are also definable.

\section{Main theorem}
\label{mt}

With our preliminaries in place, we flesh out the sketch of our main theorem from the introduction.  Throughout this section,
we fix a natural number $n$ and when we speak of a differential field we mean one with $n$ distinguished derivations.
Likewise,
when speaking of arc and jet spaces and differential fields we suppress the index $n$.  That is, we write $J_m$ for $J_{m,n}$ and $\cA_m$ for $\cA_{m,n}$.

In what follows we shall make the following hypotheses.

\begin{itemize}
\item $\RR_{\cF}$ is a fixed o-minimal structure on the reals.
\item $G$ is an algebraic group over $\CC$.
\item $Y$ is a complex algebraic variety and $G \times Y \to Y$ is a regular function
expressing a faithful action $G \curvearrowright Y$ of $G$ on $Y$.
\item $U \subseteq Y(\CC)$ is a complex submanifold of the $\CC$-points of $Y$.
\item $\Gamma < G(\CC)$ is a Zariski dense, discrete subgroup of $G(\CC)$.
\item Via the restriction of the action of $G(\CC)$ on $Y(\CC)$ to $U$, $\Gamma$ acts as a group of automorphisms of $U$.
\item $\pi:U \to X(\CC)$ is a complex analytic covering map of the algebraic variety $X$ expressing $X(\CC)$
    as the quotient $\Gamma \backslash U$.
\item $F \subseteq Y(\CC)$ is an open $\RR_{\cF}$-definable subset of $Y(\CC)$ for which the restriction
$\pi \upharpoonright F:F \to X(\CC)$ is $\RR_{\cF}$-definable and surjective onto $X(\CC)$.
\end{itemize}

\begin{Rk}
Let us note that for each $\gamma \in \Gamma$, the set $\gamma \cdot F$ is definable since the action of $G(\CC)$ on
$Y(\CC)$ is algebraic and hence \emph{a fortiori} definable.  Thus, we may cover $U$ by a set of definable sets indexed by
$\Gamma$.
\end{Rk}

\begin{Rk}
It follows from the existence of definable choice functions in o-minimal expansions of ordered fields that there is a
definable $F' \subseteq F$ so
that the restriction of $\pi$ to $F'$ is definable and a bijection between $F'$ and $X(\CC)$.  It is convenient for
us to take $F$ to be open in which case we cannot assume that the restriction of $\pi$ to $F$ is one-to-one.
\end{Rk}

From these data we construct a differential analytic map $\chi$ on $X$ which nearly inverts $\pi$.  We shall call this
resulting map the \emph{generalized logarithmic derivative associated to $\pi$}.

We begin with two lemmata on jet spaces, the first showing that jets of covering maps are themselves covering maps and the
second showing that jets of definable functions are also definable.

\begin{lem}
\label{jetcover}
For each natural number $m$, there is a natural action of $\Gamma$ on $J_{m}(U)$ and with respect to this
action, $J_{m}(\pi)$ expresses $\cA_{m}(X)(\CC)$ as $\Gamma \backslash J_{m}(U)$.
\end{lem}
\begin{proof}
Define the action of $\Gamma$ on $J_{m}(U)$ by $\gamma \cdot x := J_{m}(\gamma \cdot) (x)$.
By the functoriality of the space of jets construction, we see that for each $\gamma \in \Gamma$ we have
$J_{m}(\pi) \circ \gamma \cdot = J_{m}(\pi) \circ J_{m}(\gamma \cdot) = J_{m} (\pi \circ \gamma \cdot) = J_{m}(\pi)$.
That is, $J_{m}(\pi)$ is invariant under precomposition with the action of $\Gamma$.
Let us check now that if $x$ and $y$ are two points of $J_{m}(U)$ having the same image under $J_{m}(\pi)$, then
there is some $\gamma \in \Gamma$ with $\gamma \cdot x = y$.  Taking the images in $U$,
we see that their images $\overline{x}$ and $\overline{y}$ in $U$ have the same image under $\pi$.  Hence, there is
some $\gamma \in \Gamma$ with $\gamma \cdot \overline{x} = \overline{y}$.  As $\pi$ is
a covering map, we can find a neighborhood $V \ni \overline{y}$
on which $\pi \upharpoonright V: V \to \pi(V) \subseteq X(\CC)$ is biholomorphic.  By functoriality of $J_{m}$,
$J_{m}(\pi) \upharpoonright J_{m} (V): J_{m}(V) \to J_{m} (\pi(V)) \subseteq \cA_{m} X(\CC)$ is also
biholomorphic.  In particular, because $J_{m}(\pi) \upharpoonright J_{m}(V) (y) = J_{m}(\pi)(y) =
J_{m}(\pi)(x) = J_{m}(\pi)(\gamma \cdot x) = J_{m}(\pi) \upharpoonright J_{m}(V) (\gamma \cdot x)$, we
must have $y = \gamma \cdot x$.
\end{proof}

With our second result on jets we note that the jet of a definable function is itself definable.

\begin{lem}
\label{jetdef}
Fix an o-minimal expansion $\RR_\cF$ of the real field.
Let $Z$ be a complex algebraic variety, $V \subseteq Z(\CC)$ an open subset of $Z(\CC)$,
$W$ a complex algebraic variety, $m$  a natural number and $f:V \to W(\CC)$ a definable
analytic function.  Then $J_{m}(f):J_{m}(V) \to \cA_{m} W (\CC)$ is a definable analytic function.
\end{lem}
\begin{proof}
For the sake of legibility, we suppress the subscripts from $J$ and $\cA$.  Let us write $\nu:\cA Z \to Z$ for
the natural projection map.

Analyticity of $J(f)$ is a general feature of the jet of an analytic function (see~\cite{Bourbaki}).  Definability
of $J(f)$ follows from the facts that for any definable complex analytic function its complex derivatives are definable and
the jet space $J(V) = \nu^{-1} V$ is definable as a subset of $\cA Z (\CC)$.   Indeed, the usual limit definition of a
complex derivative may be naturally expressed using the formula defining $f$ and the field operations.  As
$f$ is analytic, it follows that all of its derivatives of all orders are definable.  Read in coordinates,
the map $J(f)$ is given by the usual Fa\`{a} di Bruni formula which is polynomial in the derivatives of $f$ up to
order $m$ and the coordinates on the jet space.

\end{proof}

\begin{Rk}
One could generalize Lemma~\ref{jetdef} to the case that $V \subseteq Z(\CC)$ is merely a complex submanifold by noting that
the jet space of $V$ at a point as a submanifold of $\cA Z(\CC)$ may be identified by iterating the standard definitions of
tangent spaces from calculus.
\end{Rk}

The relation on $Y$ defined by $x \sim y$ if and only if $(\exists g \in G(\CC)) g \cdot x = y$ is an equivalence relation
which on $\CC$-points identifies everything, but when interpreted in a differential field $M$ with field of constants $\CC$
is a nontrivial definable equivalence relation.  In particular, when $\UU$ is a differentially closed field with field of
constants $\CC$, Theorem~\ref{eidcf} says that there is a differential constructible function $\widetilde{\chi}$
defined on $Y$
having the property that for $a, b \in Y(\UU)$ one has $\widetilde{\chi}(a) = \widetilde{\chi}(b)$
 if and only if $(\exists g \in G(\CC)) g \cdot a = b$.
Since $\widetilde{\chi}$ is differential constructible, if $\CC \subseteq M \subseteq \UU$ is some intermediate differential field, it is
still the case that for $a, b \in Y(M)$ the equality $\widetilde{\chi}(a) = \widetilde{\chi}(b)$
 holds if and only if $(\exists g \in G(\CC)) g \cdot a = b$.

\begin{Rk}
In the case that $Y = {\mathbb P}^1$ and $G = \operatorname{PGL}_2$ acting by fractional linear transformations, then the function
$\widetilde{\chi}$ may be identified with the classical Schwarzian derivative.  Similar maps, sometimes with explicit formulae,
appear in the literature for other quotients of algebraic varieties by the constant points of an algebraic group (see, for example,
the generalized Schwarzians for $Y = {\mathbb P}^n$ and $G = \operatorname{PGL}_{n+1}$ in~\cite{Buium} or~\cite{SasakiYoshida}).
In analogy with the classical Schwarzian, we shall call the map $\widetilde{\chi}$ a \emph{generalized Schwarzian derivative}.
\end{Rk}

Our generalized Schwarzian $\widetilde{\chi}$  admits a more algebraic description.  Being a differential
constructible function, there is some $m \in \ZZ_+$ and a constructible function $\widetilde{\chi}_m$ defined on $\tau_m Y$ for which
$\widetilde{\chi} = \widetilde{\chi}_m \circ \nabla_m$.   The algebraic group $G$ acts algebraically on $\tau_m Y = \cA_{m} Y$
for precisely the same reason that
$\Gamma$ acts on $J_m U$.  While quotients in the category of algebraic varieties with regular maps as morphisms
do not always exist, they do always exist constructibly.
That is, there is some constructible function $\xi_m$ on $\cA_{m} Y$ expressing the quotient $G \backslash \cA_{m} Y$.  We aim to show that
for $m \gg 0$ one may take $\widetilde{\chi}_m = \xi_m$ so that $\widetilde{\chi} = \xi_m \circ \nabla_m$.

To this end we first prove two basic lemmata on differential algebraic geometry, one result about
definable equivalence relations and a second about the connection between the Kolchin topology on
a variety and the Zariski topology on its prolongations.

\begin{lem}
\label{eqclose} Let $\UU$ be a differentially closed field.
If $X$ is a differential variety over $\UU$ and $\sim$ is a definable equivalence relation on
$X$, then for $a$ and $b$ elements of $X(\UU)$, one has $a \sim b$ if and only in
$\overline{[a]_\sim}^{\text{Kolchin}} = \overline{[b]_\sim}^{\text{Kolchin}}$  where
we write $\overline{A}^\text{Kolchin}$ for the Kolchin closure of the set $A$.
\end{lem}
\begin{proof}
Certainly, if $a \sim b$, then $[a]_\sim = [b]_\sim$ so that $\overline{[a]_\sim}^{\text{Kolchin}} = \overline{[b]_\sim}^{\text{Kolchin}}$.
Conversely, by quantifier elimination in differentially closed fields we know that the sets $[a]_\sim$ and $[b]_\sim$
are differentially constructible.  Hence, there is a (Kolchin) dense open in $\overline{[a]_\sim}^{\text{Kolchin}}$ set
$V \subseteq [a]_\sim$ and a dense open in $\overline{[b]_\sim}^\text{Kolchin}$ set $W \subseteq [b]_\sim$.  If
$\overline{[a]_\sim}^\text{Kolchin} = \overline{[b]_\sim}^\text{Kolchin}$, then $V \cap W$ is a dense open subset of
$\overline{[a]_\sim}^\text{Kolchin}$.  In particular, it is nonempty so that $[a]_\sim \cap [b]_\sim \supseteq V \cap W$ is
non-empty implying that $[a]_\sim = [b]_\sim$.
\end{proof}

The Kolchin topology and the Zariski topology are related through the prolongation spaces.  With the next lemma
we show how to compute a Kolchin closure via Zariski closures in prolongation spaces.

\begin{lem}
Let $\UU$ be a differentially closed field and $X$ be an algebraic variety over $\UU$ and let $A \subseteq X(\UU)$ be a set of $\UU$-points on $X$.   Then
$$\overline{A}^\text{Kolchin} = \bigcap_{m \geq 0} \nabla_m^{-1} (\overline{\nabla_m(A)}^\text{Zariski}) \text{ .}$$
Moreover, this set may be realized as $\nabla_M^{-1} (\overline{\nabla_M(A)}^\text{Zariski})$ for $M$ sufficiently large.
\end{lem}

\begin{proof}
Each of the sets $\nabla_m^{-1} (\overline{\nabla_m(A)}^\text{Zariski})$ is a Kolchin closed set containing $A$.  Hence,
$\overline{A}^\text{Kolchin}$ is contained in $\bigcap_{m \geq 0} \nabla_m^{-1} (\overline{\nabla_m(A)}^\text{Zariski})$.

Before embarking on the proof of the other inclusion let us note that the intersection defining the righthand side of
our purported equation is actually a descending intersection.  That is, for each $m$ we have
$\nabla_{m+1}^{-1} (\overline{\nabla_{m+1}(A)}^\text{Zariski}) \subseteq
\nabla_m^{-1} (\overline{\nabla_m(A)}^\text{Zariski})$.  Indeed, the prolongation spaces form a projective
system and the $\nabla$ maps respect this system in the sense that if $\nu_{m+1,m}:\tau_{m+1}X \to \tau_m X$ is
the projection map from level $m+1$ to level $m$, then $\nabla_m = \pi_{m+1,m} \circ \nabla_{m+1}$.
Thus, if $Y \subseteq \tau_m X$ is a Zariski closed set containing $\nabla_m(A)$, then $\nu_{m+1,m}^{-1}(Y)$ is a
Zariski closed subset of $\tau_{m+1} X$ containing $\nabla_{m+1}(A)$ so that
$$\nabla_{m+1}^{-1} (\overline{\nabla_{m+1}(A)}^\text{Zariski}) \subseteq
\nabla_{m+1}^{-1} (\nu_{m+1,m}^{-1} \overline{\nabla_m(A)}^{\text{Zariski}}) = \nabla_m^{-1} \overline{\nabla_m(A)}^\text{Zariski} \text{ .}$$

For the inclusion of the righthand side in the left, consider a Kolchin closed set $Y \subseteq X$ which contains $A$.
By the definition of the Kolchin topology, there is some $m$ and a Zariski closed $Y_m \subseteq \tau_m X$
with $Y = \nabla_m^{-1} Y_m$.   Since $A \subseteq Y(\UU)$, we have $\nabla_m(A) \subseteq Y_m$.  Thus,
$\overline{\nabla_m(A)}^\text{Zariski} \subseteq Y_m$ so that $\nabla_m^{-1} (\overline{\nabla_m(A)}^\text{Zariski}) \subseteq Y$.
Clearly,  for any $\ell \geq m$ if we set $Y_\ell := \nu_{\ell,m}^{-1} Y_m$, then
$Y = \nabla_\ell^{-1} Y_\ell$.  Thus, for all $\ell \geq m$ we have $\nabla_\ell^{-1} (\overline{\nabla_\ell(A)}^\text{Zariski}) \subseteq Y$.
On the other hand, by the observation above, we have $\bigcap_{M \geq 0} \nabla_M^{-1} (\overline{\nabla_M(A)}^\text{Zariski})
= \bigcap_{M \geq m} \nabla_M^{-1} (\overline{\nabla_M(A)}^\text{Zariski}) \subseteq Y$.

The moreover clause is an immediate consequence of the noetherianity of the Kolchin topology.
\end{proof}

We employ the above lemmata to relate the quotient of an algebraic variety by the constant points of an algebraic group
to algebraic quotients.

\begin{prop}
\label{difftoalgquot}
Let $\UU$ be a differentially closed field, $G$  an algebraic group over
the constant field $C = \UU^\Delta$, $X$ an algebraic variety over $C$ and $G \curvearrowright X$ an
action of $G$ on $X$ over $C$.   For $\UU$-points $a, b \in X(\UU)$ the following are equivalent.
\begin{itemize}
\item There is some $g \in G(C)$ with $g \cdot a = b$.
\item For all $m \in \NN$ there is some $g_m \in G(\UU)$ with $g_m \cdot \nabla_m (a) = \nabla_m(b)$.
\end{itemize}

Moreover, there is some natural number $M$ so that these conditions are equivalent to
\begin{itemize}
\item There is some $g_M \in G(\UU)$ with $g_M \cdot \nabla_M(a) = \nabla_M(b)$.
\end{itemize}
\end{prop}

\begin{proof}
For this proof, we write $\overline{A}$ for the Zariski closure of a set $A$ and $\overline{A}^\text{Kolchin}$ for its
Kolchin closure.

Define the equivalence relation $\sim$ on $X(\UU)$ by
$$a \sim b :\Longleftrightarrow (\exists g \in G(C)) g \cdot a = b$$ and
let $\approx$ be given by
$$a \approx b :\Longleftrightarrow (\forall m \in \NN) (\exists g_m \in G(\UU)) g_m \cdot \nabla_m (a) = \nabla_m(b) \text{ .}$$

Note that $\sim$ is a definable equivalence relation while \emph{prima facie} $\approx$ is merely type-definable.   Clearly, $a \sim b
\Rightarrow a \approx b$ for if $g \in G(C)$ satisfies $g \cdot a = b$, then for each $m$ we may take $g_m := \nabla_m(g)$ (which belongs to the
image of the zero section as the constant points are precisely those whose image under $\nabla_m$ agree with the zero section) to
witness that there is some $g_m \in G(\UU)$ with $g_m \cdot \nabla_m(a) = \nabla_m(g) \cdot \nabla_m(a) = \nabla_m(g \cdot a) =
\nabla_m(b)$.

On the other hand, we observe that for each $m$ the Zariski closure of $\nabla_m ([a]_\sim)$ is the Zariski closure of the
$G(\UU)$-orbit of $\nabla_m(a)$.  Indeed, consider the algebraic group $S$ defined as the stabilizer of
$\overline{\nabla_m([a]_\sim)}$ in $G$:

$$S(\UU) :=  \{ g \in G(\UU) ~:~ g \cdot \overline{\nabla_m([a]_\sim)} \subseteq \overline{\nabla_m([a]_\sim)} \} $$

As $\nabla_m ([a]_\sim)$ is a homogeneous space for $G(C)$ under the natural action of $G(\UU)$ on $\tau_m X$, we see that
$G(C) \subseteq S(\UU) \subseteq G(\UU)$.  As $G(C)$ is Zariski dense in $G(\UU)$ and $S$ is closed, we conclude that $S = G$.
Thus, $\overline{\nabla_m([a]_\sim)} \supseteq  \overline{ G \cdot \nabla_m(a)}$.   On the other hand, since $\nabla_m([a]_\sim)
= G(C) \cdot \nabla_m(a) \subseteq G(\UU) \cdot \nabla_m(a)$, we must have equality.

Therefore, if $a \approx b$, then for each $m \in \NN$ by definition over $\approx$ we have that $G(\UU) \cdot \nabla_m (a) =
G(\UU) \cdot \nabla_m(b)$ so that  $\overline{\nabla_m([a]_\sim)} = \overline{ G(\UU) \cdot \nabla_m(a)} = \overline{ G(\UU) \cdot \nabla_m(b)}
= \overline{[b]_\sim}$.  As this equality is true for every $m$, $\overline{[a]_\sim}^{\text{Kolchin}} = \overline{[b]_\sim}^{\text{Kolchin}}$.
By Lemma~\ref{eqclose}, $a \sim b$.

The moreover clause follows by the compactness theorem.
\end{proof}

It follows from elimination of imaginaries in algebraically closed fields that if the algebraic group $H$ acts on an
algebraic variety $V$, then the quotient $V \to H \backslash V$ may be realized constructibly.  In fact, while the map expressing the
quotient cannot be taken to be regular, it may be assumed to have some regularity properties.

\begin{lem}
\label{regquot}
Let $K$ be an algebraically closed field, $H$ an algebraic group over $K$, $V$ an algebraic variety over $K$, and
$H \curvearrowright V$ an action of $H$ on $V$, also defined over $K$.  Then there are
\begin{itemize}
\item a chain of closed algebraic varieties  $\varnothing = V_0 \subsetneq V_1 \subsetneq \cdots \subsetneq V_m$,
\item an algebraic variety $W$, and
\item regular maps $\xi_i:V_i \smallsetminus V_{i-1} \to W$ for each positive $i \leq m$
\end{itemize}

so that
\begin{itemize}
\item each $V_i$ is $H$-invariant and
\item for $a, b \in (V_i \smallsetminus V_{i-1})(K)$ one has $\xi_i(a) = \xi_i(b) \Longleftrightarrow (\exists g \in H(K)) ~ g \cdot a = b$.
\end{itemize}
\end{lem}

\begin{proof}
We work by noetherian induction on $V$ with the case that $\dim(V) = 0$ being immediate.
By elimination of imaginaries we find a constructible function $\psi:V \to W$ to some algebraic variety $W$ expressing $H \backslash V$.  (Note:
we are not claiming that $W = H \backslash V$.)   As $\psi$ is a constructible function, we can find a dense open $U \subseteq V$ for which
$\psi \upharpoonright U$ is regular.  Set $U' := H \cdot U$ which is again a dense open subset of $V$.  We claim that $\psi \upharpoonright U'$
is regular.  Indeed, we may cover $U'$ with charts of the form $g S$ where $S \subseteq U$ is an open affine in $U$ and $g \in G(K)$.
By the $H$-invariance of
$\psi$, we see that on $g S$, $\psi$ satisfies $\psi(gx) = \psi(x)$.  That is, $\psi \upharpoonright gS$ agrees with $\psi \upharpoonright S$
via the isomorphism $g\cdot:S \to gS$.  Thus, $\psi \upharpoonright gS$ is regular and therefore $\psi \upharpoonright U'$ is regular.  Set
$V' := V \smallsetminus U'$.  By induction, $V'$ admits the requisite chain, of length $m$, say.  Let $V_{m+1} := X$.
\end{proof}

We are now in a position to prove an algebraic counterpart of our main theorem.

\begin{prop}
\label{mainthmalg}
With the notation and hypotheses as introduced at the beginning of this section,
if $\eta:Y \to Z$ is a constructible function from $Y$ to some algebraic variety $Z$ expressing the quotient $Y \to G \backslash Y$
as in Lemma~\ref{regquot}, 
then $\chi:X \to Z$ defined by $\chi := \eta \circ \pi^{-1}$ (for any choice of a branch of $\pi^{-1}$)
is a constructible function. 
\end{prop}

\begin{proof}   For this proof, by a differential field $M$ we mean one with field of constants $M^\Delta = \CC$.

Consider the following set.

$$\Xi := \{ (x,z) \in X(\CC) \times Z(\CC) ~:~ (\exists y \in Y (\CC) ) ~ \pi (y) = x ~\&~ \eta (y) = z \}$$

Let us observe that the set $\Xi$ is definable.
Indeed, because $\pi$ is $\Gamma$-invariant and $\eta$ is $G$-invariant (hence,
also $\Gamma$-invariant), in the definition of $\Xi$ we may restrict $y$ to $F$.   That is, we have

$$\Xi =  \{ (x,z) \in X(\CC) \times Z(\CC) ~:~ (\exists y \in F ) ~ \pi (y) = x ~\&~ \eta(y) = z \} \text{ .}$$

As $\pi \upharpoonright F$ is definable, this expression presents $\Xi$ as a definable set.

Not only is the set $\Xi$ definable, but it is the graph of a function.  Indeed,  using definable choice in $\RR_\cF$,
we see that there is a definable function $\zeta:X(\CC) \to F \subseteq Y (\CC)$ which is a right inverse to $\pi$.  Again using
the $\Gamma$-invariance of $\eta$ we see that $\Xi$  is the graph of $\eta \circ \zeta$.    Let
us write $\chi:X \to Z$ for the function whose graph is $\Xi$.  Let us note that while this definition of $\chi$
expresses its definability, there are other ways it could be presented.  Indeed,
from the $\Gamma$-invariance of $\eta$ for any $x \in X(\CC)$ we have
$\chi(x) = \eta (\hat{\zeta}(x))$ where $\hat{\zeta}$ is \emph{any} branch of $\pi^{-1}$ near $x$.

By Lemma~\ref{regquot}, there is a sequence of a closed subvarieties $\varnothing = V_0 \subsetneq V_1 \subsetneq
\cdots \subsetneq V_\ell = Y$ so that each $V_i$ is $G$-invariant and the restriction of
$\eta$ to $V_i \smallsetminus V_{i-1}$ is regular.   For each $i \leq \ell$, define
$\xi_i := (\eta \upharpoonright V_i) \circ \zeta: X \to Z$.   Note that $\xi_\ell = \chi$.
We show by induction on $i$ that
$\xi_i$ is constructible.     The case of $i = 0$ is trivial.  Let us consider now the case that $i > 0$.
Let us define $W_i := \eta (V_i) \subseteq Z$.
Since $\eta$ is a constructible function, $W_i$ is a constructible subset of $Z$.
Consider $\Xi_i := \Xi \cap (X \times (W_i \smallsetminus W_{i-1}))$.  The set $\Xi_i$ being the intersection of
two definable sets is definable.   Moreover, our second presentation of $\chi$ shows that
$\Xi_i$ is the graph of an analytic function.  Indeed, if $(x,y) \in \Xi_i$, then fix a branch $\hat{\zeta}$
of $\pi^{-1}$ near $x$.  We have $y = \eta (\hat{\zeta} (x))$ and, in fact, near $(x,y)$,
$\Xi_i$ is the graph of $\eta \circ \hat{\zeta}$.   As $y \in W_i \smallsetminus W_{i-1}$ and
$V_i$ and $V_{i-1}$ are $G$-invariant, necessarily $\hat{\zeta}(x) \in (V_i \smallsetminus V_{i-1})(\CC)$.
Thus, near $(x,y)$, $\Xi_i$ is the graph of the analytic function $(\eta \upharpoonright (V_i \smallsetminus V_{i-1}))
\circ \hat{\zeta}$.   By the Peterzil-Starchenko o-minimal GAGA Theorem~\ref{PSGAGA}, $\Xi_i$ is algebraic.  By induction,
$\xi_{i-1}$ is algebraically constructible and the graph of $\xi_i$ is simply the union of $\Xi_i$ and the graph
of $\xi_{i-1}$.  Hence, $\xi_i$ is itself algebraically constructible.

Taking $i = \ell$, we see that $\chi$ is algebraically constructible.
\end{proof}

We deduce our main theorem from Proposition~\ref{mainthmalg}.

\begin{theorem}
\label{mainthm}
With the notation and hypotheses as introduced at the beginning of this section,
if $\widetilde{\chi}:Y \to Z$ is a generalized Schwarzian for the action of $G^\Delta$ on $Y$, that is, it
is a differential constructible function from $Y$ to some algebraic variety $Z$ expressing the quotient $Y \to G^\Delta \backslash Y$,
then $\chi:X \to Z$ defined by $\chi := \widetilde{\chi} \circ \pi^{-1}$ (for any choice of a branch of $\pi^{-1}$)
is a differential constructible function, which we shall call the generalized logarithmic derivative associated to $\pi$.
\end{theorem}

\begin{proof}

By Proposition~\ref{difftoalgquot} for $N \gg 0$ there is a constructible function $\widetilde{\chi}_N:\cA_N Y \to Z$
so that $\widetilde{\chi}:Y \to Z$ takes the form $\widetilde{\chi} := \widetilde{\chi}_N \circ \nabla_N$ and
for any differential field $M$ with field of constants $\CC$ one has that
\begin{itemize}
\item for points $a, b \in Y(M)$
$$\widetilde{\chi}(a) = \widetilde{\chi}(b) \Longleftrightarrow (\exists g \in G(\CC)) ~ g \cdot a = b$$
\item for points $a, b \in \cA_N Y(M)$
$$\widetilde{\chi}_N(a) = \widetilde{\chi}_N(b) \Longleftrightarrow (\exists g \in G(M^{\operatorname{alg}})) ~ g \cdot a = b \text{ .}$$
\end{itemize}

With this choice of $N$, we may realize our generalized logarithmic derivative $\chi:X \to Z$ as $\chi = \widetilde{\chi} \circ \pi^{-1} =
\widetilde{\chi}_N \circ J_N(\pi)^{-1} \circ \nabla_N$ as in the following diagram.

$$
\xymatrix{Y \ar@/^/[rr]^{\nabla_N}  & & \cA_N(Y) \ar[ll] \ar[rr]^{\widetilde{\chi}_N} & & Z \\
U \ar@{_{(}->}[u] \ar[d]^{\pi} & & J_N(U) \ar@{_{(}->}[u] \ar[ll]  \ar[d]^{J_N(\pi)} & &  \\
X \ar@/^/[rr]^{\nabla_N} \ar@/_4pc/[rrrruu]_{\chi} \ar@{.>}@/^/[u]^{\pi^{-1}} & & \cA_N(X) \ar[ll] \ar@/^/@{.>}[u]^{J_N(\pi)^{-1}} \ar[uurr]^{\chi_N} \\
}
$$

Indeed, the hypotheses of Proposition~\ref{mainthmalg} apply with $\cA_N Y$ in place of $Y$,
$J_N(U)$ in place of $U$, $J_N(F)$ in place of $F$, $J_N(\pi)$ in place of $\pi$, $\cA_N X$ in
place of $X$, and $\widetilde{\chi}_N$ in place of $\eta$.  Thus, by Proposition~\ref{mainthmalg},
$\chi_N := \widetilde{\chi}_N \circ J_N(\pi)^{-1}:\cA_N X \to Z$ is constructible.   Hence,
$\chi = \chi_N \circ \nabla_N:X \to Z$ is differential constructible.

\end{proof}

\begin{Rk}
The construction of $\chi$ as a differential meromorphic function is classical.  The new content of Theorem~\ref{mainthm} is that
$\chi$ is differential \emph{constructible} and that it is defined everywhere on $X$.
\end{Rk}

\begin{Rk}
The function $\chi$ depends on the choice of $\widetilde{\chi}$, which is itself well-defined only up to differential constructible
isomorphism.  However, the equivalence relation defined by $a \sim b :\Longleftrightarrow \chi(a) = \chi(b)$ is intrinsic.
\end{Rk}

\section{Towards a theory of generalized logarithmic derivatives}
\label{theorygld}

Using our construction of the generalized logarithmic derivatives it is possible to deduce differential
algebraic properties of these maps from their analytic interpretation.  In this section we draw some of these
conclusions though we leave a fine analysis for a future work.

Throughout this section we shall work inside a differentially closed field $(\UU,\Delta)$ which we shall realize as a differential
subfield of some differential field $M$ of germs of meromorphic functions.

We shall work with the conventions and notation of Section~\ref{mt}.  By way of notation, for $a \in X(\UU)$ we denote by
$F_a$ the differential subvariety of $X$ defined by $\chi(x) = \chi(a)$.

Let us begin with a simple observation about the meaning of two points having the same image under $\chi$.

\begin{prop}
\label{gachi}
For $a, b \in X(\UU)$, one has $\chi(a) = \chi(b)$ if and only if for any choice of a branch of $\pi^{-1}$
there is some $g \in G(\CC)$ so that $\pi(g \cdot \pi^{-1}(a)) = b$.
\end{prop}
\begin{proof}
By construction, $\chi = \widetilde{\chi}  \circ \pi^{-1}$ where $\widetilde{\chi}(x) = \widetilde{\chi}(y)$ if and
only if there is some $g \in G(M^\Delta) = G(\CC)$ with $g \cdot x = y$.  Hence, $\chi(a) = \chi(b)$ if and only if
there is some $g \in G(\CC)$ with $g \cdot \pi^{-1}(a) = \pi^{-1} (b)$, or equivalently, $\pi(g \cdot \pi^{-1}(a)) = b$.
\end{proof}

Using again the analytic interpretation of $\chi$ we compute the dimension of a fibre of $\chi$.

\begin{prop}
Let $a \in X(\UU)$.
Then $\dim(F_a) \leq \dim(G^\Delta)$ (which is the dimension of $G$ as an algebraic group).
In fact, if $\pi^{-1}$ is any branch of the inverse of $\pi$, then recalling that $\UU$ is a field of germs of meromorphic
functions, $\dim(F_a) = \dim(G^\Delta) - \dim(G^\Delta_{\pi^{-1}(a)})$, where
$G^\Delta_{\pi^{-1}(a)}$ is the stabilizer of $\pi^{-1}(a)$ in $G^\Delta$.
\end{prop}

\begin{proof}
As in the proof of Theorem~\ref{mainthm}, we take $N \gg 0$ large enough
so that $\chi = \widetilde{\chi}_N \circ J_N(\pi)^{-1} \circ \nabla_N$ where $\widetilde{\chi}_N:\cA_N Y \to Z$
is an algebraically constructible map expressing the quotient $\cA_N Y \to G \backslash \cA_N Y$.
Then the differential variety $F_a$ is the pullback by $\nabla_N$ of the algebraic subvariety $(F_a)_N$ of $\cA_N X$
defined by $\chi_N(x) = \chi_N(\nabla_N(a))$.  The preimage $\widetilde{F}_a$ of $(F_a)_N$ in $J_N(U)$ is
a complex analytic variety which contains $G(\CC) \cdot J_N(\pi)^{-1}(\nabla_N(a))$.  As the map $J_N(\pi)$
is a covering map, the analytic variety $\widetilde{F}_a$ has the same dimension as that of $(F_a)_N$.
Since the fibre of $\widetilde{\chi}_N$ over $\chi(a)$ is precisely the $G(\CC)$ orbit of
$\nabla_N \pi^{-1} (a) = J_N(\pi)^{-1} (\nabla_N(a))$, we see that
$\widetilde{F}_a = \widetilde{\chi}_N^{-1} (\chi(a))$.  Hence, the dimension of $F_a$ is equal to the
dimension of the orbit of $\nabla_N (\pi^{-1}(a)) = \dim G^\Delta - \dim G^\Delta_{\pi^{-1}(a)}$.
\end{proof}

\begin{Rk}
A similar calculation occurs in~\cite{BerZud} in which Bertrand and Zudilin show that the partial 
differential fields generated by Siegel modular forms have transcendence degrees governed by the 
groups acting on the covering spaces.
\end{Rk}

Our calculation of the dimension of the fibres suggests an alternate method to describe the fibres of the generalized logarithmic
derivative.  Let us introduce the notion of a generalized Hecke correspondence.

\begin{Def}
The commensurability group of $\Gamma$ is the subgroup of $G(\CC)$ defined by
$$\Gamma^{\text{comm}} := \{ \gamma \in G(\CC) ~:~ \gamma(U) = U ~\&~ [\Gamma:\Gamma^\gamma \cap \Gamma] < \infty ~\&~ [\Gamma^\gamma:\Gamma^\gamma \cap \Gamma] < \infty \} \text{ .}$$
For $\gamma \in \Gamma^{\text{comm}}$, the image of the graph of $\gamma \cdot:U \to U$ under $(\pi,\pi)$ is an algebraic correspondence $T_\gamma \subseteq
X \times X$ which we shall call the generalized Hecke correspondence associated to $\gamma$.
\end{Def}

\begin{Rk}
That the generalized Hecke correspondences are algebraic is well-known, but it may also be seen as a consequence of Theorem~\ref{PSGAGA}. 
\end{Rk}

Let us note that the fibres of $\chi$ are closed under the generalized Hecke operators.

\begin{prop}
If $a, b \in X(\UU)$ and there is some $\gamma \in \Gamma^{\text{comm}}$ with $(a,b) \in T_\gamma(\UU)$, then
$\chi(a) = \chi(b)$.
\end{prop}
\begin{proof}
By the construction of $T_\gamma$, for some choice of a branch of $\pi^{-1}$ and for some $\delta \in \Gamma$,
we have $\gamma \cdot \pi^{-1}(a) = \delta \cdot \pi^{-1} (b)$.  Hence, $\pi^{-1}(a)$ and $\pi^{-1}(a)$ are in the $G(\CC)$ orbit so
that $\widetilde{\chi}(\pi^{-1}(a)) = \widetilde{\chi}(\pi^{-1}(b))$ implying that $\chi(a) = \chi(b)$.
\end{proof}

If the commensurability group is large, then the fibres of $\chi$ are precisely the Kolchin closures of the generalized Hecke orbits.

\begin{prop}
Assume that $\Gamma^{\text{comm}}$ is large in the sense that every set of coset representatives of $\Gamma$ in $\Gamma^{\text{comm}}$ is
Zariski dense in $G(\CC)$.
Then for each $a \in X(\UU)$, the fibre $F_a$ is the Kolchin closure of the generalized Hecke orbit of $a$.
\end{prop}
\begin{proof}
Let $a \in X(\UU)$ and let $V \subseteq X$ be the Kolchin closure of the generalized Hecke orbit of $a$.  Let
$N \gg 0$ be large enough so that $\chi = \chi_N \circ \nabla_N$ and there is an algebraic variety $V_N \subseteq \cA_N X$
with $V(\UU) = \nabla_N^{-1}(V_N(\UU))$ and $\nabla_N(V(\UU))$ is Zariski dense in $V_N$.   Let $\widetilde{V}_N$ be
a component of $J_N^{-1} V_N$.  Let $H := \{ \gamma \in G(\CC) ~:~ \gamma \cdot \widetilde{V}_N = \widetilde{V}_N \}$.
Since $J_N(\pi)^{-1} V_N$ is closed under the action of $\Gamma^{\text{comm}}$, we see that $H$ contains a set of coset representatives
for $\Gamma$ in $\Gamma^{\text{comm}}$.  By our hypothesis on $\Gamma^{\text{comm}}$, $H$ is dense in $G(\CC)$, and, hence, is equal to $G(\CC)$.
Thus, $\widetilde{V}_N = \widetilde{\chi}_N^{-1} (\chi(a))$ so that $V$ is defined by $\chi(x) = \chi(a)$.
\end{proof}

One might ask for a description of the set of points in $F_a$ which are algebraic over $a$.  We solve this problem only for the case that
$a$ is the generic point (in the sense of the Zariski topology) of $X$ and that $\CC(a) = \CC(X)$ is given a differential structure
for which $\CC(X)^\Delta = \CC$.

\begin{prop}
Let $K = \CC(X)$, the field of rational functions on $X$, given with a basis of $\CC$-derivations $\Delta = \{ \partial_1,
\ldots, \partial_n \}$. Fix an embedding of $K$ into $\UU$ over $\CC$. Let $a \in X(K)$ be the generic point, that is, the
$K$-rational point corresponding to the identity map $\operatorname{id}:X \to X$.  Then
$$F_a(K^{\text{alg}}) = \{ b \in X(\UU) ~:~ (\exists \gamma \in \Gamma^{\text{comm}}) (a,b) \in T_\gamma(\UU) \} \text{ .}$$
\end{prop}
\begin{proof}
We have already noted $F_a(\UU)$ contains all points which are in generalized Hecke correspondence with $a$ and because
each generalized Hecke correspondence is a finite-to-finite correspondence, all such points are algebraic over $K$.  Hence, we need
only verify the left to right inclusion.

Suppose that $b \in F_a(K^{\text{alg}})$.  Geometrically, we may represent $b$ by a dominant, generically finite,
rational map $b:X' \dashrightarrow X$ where $X'$ is an irreducible complex algebraic variety.
Consider some nonempty connected open set $V' \subseteq X'(\CC)$ for which the covering $\pi:U \to X(\CC)$
trivializes over $V := b(V')$ and $V \cong \BB^{\dim(X)}$ is a coordinate chart.
Let $\widetilde{V} \subseteq U$ be one of the components of $\pi^{-1} V$.  Let
us work with the branch of $\pi^{-1}$ which on $V$ takes values in $\widetilde{V}$.
By Proposition~\ref{gachi}, there is some
$g \in G(\CC)$ with $\pi^{-1}(b) = g \cdot \pi^{-1}(a)$.  Hence, $(a,b)$ (regarded now a point  in
$\operatorname{Mor}(V',V \times V) \subseteq (X \times X)(\cM(V')) \subseteq (X \times X)(\UU)$)
  lies on the analytic subset of $V \times V$, which we shall call
$\Upsilon := (\pi,\pi) (\{ (x,y) \in \widetilde{V} \times \widetilde{V} ~:~ g \cdot x = y \})$.
Let $D \subseteq X \times X$ be the algebraic locus of $(a,b)$ over $\CC$, that is,
$D$ is the smallest $\CC$-variety which contains the point $(a,b)$.   The analytic variety $\Upsilon$
being the graph of a function is irreducible.  Hence, either  $\dim(\Upsilon \cap D) < \dim(X)$ or $\Upsilon \subseteq D$.
The former condition is impossible by our hypothesis that $\CC(a)^\Delta = \CC$:  writing $a = (a_1,\ldots,a_{\dim(X)})$
and $b = (b_1,\ldots,b_{\dim(X)})$, with respect to coordinates on $V$, then the hypothesis on the
constants is equivalent to the assertion that the matrix $(\partial_i a_j)$ has rank $\dim(X)$.  A further
analytic relation would force the rank to be $\leq \dim(\Upsilon \cap D) < \dim(X)$.   Thus, $\Upsilon \subseteq D$
implying that $\{ (x,y) \in U \times U ~:~ g \cdot x = y \} \subseteq (\pi \times \pi)^{-1} D$ so that the image of the
graph of the action of $g$ is algebraic which is only possible for $g \in \Gamma^{\text{comm}}$.

\end{proof}

\section{Some covering maps with algebraic generalized logarithmic derivatives}
\label{examplessect}

In this section we specialize Theorem~\ref{mainthm} to some concrete cases of
complex algebraic varieties admitting suitable covering maps.   In each case,
we are required to check that the covering map is definable in some o-minimal
expansion of the real field on some fundamental domain.

\subsection{Logarithmic derivatives}
Let us begin with the example with which we introduced this paper.
As we have already noted, the usual logarithmic derivative
$\frac{d}{dz} \log:\Gm \to \Ga$   defined by $f \mapsto \frac{f'}{f}$ is clearly differential
algebraic and may be obtained from Theorem~\ref{mainthm} using the fact that the
restriction of $\exp:\CC \to \CC^\times$ to the set
$$F := \{ z  \in \CC ~:~ -2 \pi < \operatorname{Im}(z) < 2 \pi \}$$
is surjective and definable in $\RR_{\text{an},\exp}$ via the formula
$$\exp(z) = e^{\operatorname{Re}(z)} \cos(\operatorname{Im}(z)) + i e^{\operatorname{Re}(z)} \sin(\operatorname{Im}(z)) \text{ .}$$

Of course, in $\RR_{\text{an},\exp}$ the real exponential function is explicitly allowed as a definable function while the functions
$\cos(x)$ and $\sin(x)$ restricted to $[-2 \pi,2 \pi]$ are expressible as restricted analytic terms.

Let us recall the theory of logarithmic
derivatives for algebraic groups over the constants; for further details see~\cite{Kolchin-DAAG}.  We shall
work with differential fields $(K,\Delta)$ with $\Delta = \{ \partial_1,
\ldots, \partial_n \}$.   If $G$ is an algebraic
group defined over the constants $C = K^\Delta$, then the tangent
bundle of $G$ splits as a semi-direct product $TG = G \ltimes T_e G$ where
$T_e G$ is the tangent space to $G$ at the identity element $e \in G$.
The second part of the splitting $\eta: TG \to T_e (G)$  is given by $(g,v) \mapsto d(g^{-1} \cdot) v$
where we write the points on $T G$ as pairs $(g,v)$ consisting of a point $g$
of $G$ and a vector $v$ in the tangent space to $G$ at $g$.  Because
$G$ is defined over $C$,  we may identify $\tau_1 G$ with the $n^\text{th}$
fibre power of $T G$ over $G$  so that
$\nabla:G(K) \to \tau_1 G(K)$ gives a map of groups
$\nabla:G(K) \to (T G \times_G \cdots \times_G TG)(K)$.
The logarithmic derivative $\partial \log_G:G(M) \to (T_e G)^n(M)$ is given by
$x \mapsto \eta^{\times n}(\nabla(x))$.
The differential algebraicity of this logarithmic derivative is
clear from this construction, but as with the usual logarithmic derivative
this map may be interpreted analytically.

If $G$ is a connected complex algebraic
group, then Lie theory supplies a complex analytic exponential map
$\exp_G:T_e G(\CC) \to G(\CC)$.  In
fortuitous cases, for example when $G$ is commutative, the exponential map
is a covering map, but this is not always so.  With the next lemma, we
note that in the case that
$G$ is commutative, there is a definable (in $\RR_{\text{an},\exp}$) subset
$F \subseteq T_e G(\CC)$ for which $\exp_G \upharpoonright F$ is definable
and surjective.

\begin{lem}
\label{defexp}
If $G$ is a connected commutative complex algebraic group of dimension $g$, then there is a semialgebraic subset
$F \subseteq \CC^g \cong T_e G (\CC)$ for which $\exp_G \upharpoonright F:F \to G(\CC)$ is
surjective and definable in $\RR_{\text{an},\exp}$.
\end{lem}
\begin{proof}
By the structure theory of algebraic groups, $G$ fits into an exact sequence
$$\xymatrix{ 0 \ar[r] & L \ar[r] & G \ar[r] & A \ar[r] & 0}$$
where $L$ is a commutative, connected linear algebraic group
and $A$ is an abelian variety.  As $\CC$ has characteristic zero,
$L$ is isomorphic to a product of additive and multiplicative groups (see~\cite{Serre}).
The tangent space at the origin of the additive group may be identified
with the additive group itself and relative to this identification,
the exponential map may be taken to be the identity map.  As noted above,
the usual complex exponential function
admits a semialgebraic fundamental domain on which the exponential function restricts to an $\RR_{\text{an},\exp}$-definable
function.  Hence, we may find a semialgebraic set $F_L \subseteq T_e L(\CC) \cong \CC^{\dim(L)}$
for which $\exp_L$ restricted to $F_L$ is surjective onto $L(\CC)$ and definable in $\RR_{\text{an},\exp}$.

Choose a subspace $V \leq T_e G(\CC)$ for which
$V \cap T_e L (\CC) = 0$ and $V + T_e L(\CC) = T_e G(\CC)$.   Since the fundamental domain for $\exp_A$ is bounded, if $F_A \subseteq V$
is any closed box for which the image of $F_A$ under the natural map $T_e G (\CC) \to T_e A (\CC)$ contains
a fundamental domain, then
the restriction of $\exp_G$ to $\{ 0_{T_e L (\CC) } \} \times F_A$ is explicitly $\RR_{\text{an}}$-definable and its
image under $G \to A$ maps onto $A$.   Let us set $F := F_L \times F_A$.  Then on $F$, $\exp_G$ may be expressed as the
sum (in $G$) of $\exp_L \upharpoonright F_L$ and $\exp_G \upharpoonright \{ 0_{T_e L (\CC)} \} \times F_A$ and is
therefore definable.  A simple diagram chase shows that the restriction of $\exp_G$ to $F$ is surjective.
\end{proof}

 The generalized Schwarzian on
$T_e G(\CC) \cong \CC^g$ (where $g = \dim G$) may be given in coordinates
by differentiation.  That is, $(x_1,\ldots,x_g) \mapsto (\partial_i (x_j))_{1 \leq i \leq n, 1 \leq j \leq g}$.
Theorem~\ref{mainthm} implies the differential algebraicity
of the generalized logarithmic derivative $\chi$ given $x \mapsto (\partial_i \log_G)_{i = 1}^n \in
T_e G(\CC)^n$ where
$\log_G:G \to T_e(G)$ is any branch of $\exp_G^{-1}$.  Note that for $M = \cM(V)$
a differential field of meromorphic functions on some connected complex manifold $V$
with $\dim V = n$, the fibres of $\chi$ and of $d \log_G$ agree.  Indeed,
from the analytic presentation we see that $\chi(f) = \chi(g)$ if and only
if there is some $c \in T_e G(\CC)$ with $g = \exp_G( c + \log_G(f)) =
\exp_G(c) + f$.  Since $\exp_G$ is surjective on complex points, we have that
$\chi(f) = \chi(g)$ if and only if $f - g \in G(\CC)$.

Curiously, because it is not always the case the Lie exponential map is a
covering map, it is not always possible to realize Kolchin's general
logarithmic derivative for connected
algebraic groups over $\CC$ via Theorem~\ref{mainthm}.

\subsection{Moduli spaces of abelian varieties and general arithmetic varieties}
Modular functions are the quintessential covering maps to which
Theorem~\ref{mainthm} apply and they generalize naturally to the
covering maps for moduli spaces of abelian varieties.

\subsubsection{Moduli of elliptic curves}
Let us recall some of the basic theory of modular functions and
moduli spaces of elliptic curves.  The reader may wish to consult~\cite{milneMF}
for more details.
We write the upper half plane as
$\fh := \{ z \in \CC ~:~ \operatorname{Im}(z) > 0 \}$
and regard $\fh$ as an open subset of ${\mathbb P}^1(\CC)$ on which the
algebraic group $\operatorname{GL}_2$ acts by linear fractional transformations.
Via this action, $\fh$ is preserved by $\operatorname{GL}_2^+(\RR)$,
the subgroup of $\operatorname{GL}_2(\RR)$ on which the determinant is
positive.   In particular, $\fh$ is preserved by $\operatorname{SL}_2(\ZZ)$.

For each $N \in \ZZ_+$, let $\Gamma(N)$ be the kernel of
the reduction modulo $N$ map $\operatorname{SL}_2(\ZZ) \to \operatorname{SL}_2(\ZZ/N\ZZ)$.
The quotient $\Gamma(N) \backslash \fh$ has the structure of
an affine algebraic variety which is usually called $Y(N)(\CC)$.
When $N > 2$, the quotient map $j_N:\fh \to Y(N)(\CC)$ is a covering map,
but for $N = 1$, $j_1$, which is usually denoted by $j$, the usual analytic
$j$-function, fails to be a covering map because it is ramified over two points.
The curve $Y(1)$ may be identified with $\AA^1$ and with the usual normalizations,
the two points over which $j$ is ramified are $0$ and $1728$ (with preimages
$\exp(\frac{\pi i}{3})$ and $i$, respectively).

If one takes $F := \{ z \in \fh ~:~ |\operatorname{Re}(z)| \leq \frac{1}{2}
\text{ and } \operatorname{Im}(z) \geq \frac{\sqrt{3}}{2} \}$, then
$F$ contains a fundamental domain for $j$ and the restriction of $j$ to
$F$ is definable in $\RR_{\text{an},\exp}$.  Indeed, the function
$\tau \mapsto q_\tau := \exp(2 \pi i \tau)$ maps $F$ to the disc of radius $\exp(- \pi \sqrt{3})$.
It is known that the $q$-expansion of the $j$-function is meromorphic on the unit disc with a
simple pole at the origin.   Hence, $j$ may be expressed as the composition of
the restriction of a meromorphic function to
the closed disc of radius $\exp( - \pi \sqrt{3})$ with $\exp(2 \pi i \tau)$, which as
we have noted above, is definable in $\RR_{\text{an},\exp}$ on vertical strips of bounded width.
As each of the modular functions $j_N$ factor through $j$, they, too, are definable when
restricted to $F$.

Consequently, for $N > 1$, Theorem~\ref{mainthm} applies directly showing that the
generalized logarithmic derivative $\chi:Y_0(N) \to Z$ defined as the usual Schwarzian
derivative appied to $j_N^{-1}$ is differential constructible.  For the $j$-function
itself, we may use the flexibility provided by the hypotheses of Theorem~\ref{mainthm}:
take $U$ to be the open subset of $\fh$ on which $j$ is a covering map and $X$ to be
$\AA^1 \smallsetminus \{ 0, 1728 \}$.

\subsubsection{Universal abelian schemes}
The definability of $j$ on a fundamental domain is part of a theorem of Peterzil and Starchenko
to the effect that the two-variable Weierstra{\ss} $\wp$-function is definable (in $\RR_{\text{an},\exp}$)
on a fundamental domain~\cite{PS-wp}.  That is, the covering maps associated to universal families of
elliptic curves may be treated by Theorem~\ref{mainthm}.  In~\cite{PS-theta}, Peterzil and Starchenko extend this
result showing
that the covering maps associated to the universal families of abelian varieties are definable in $\RR_{\text{an},\exp}$.

Let $g \in \ZZ_+$ be a positive integer and write $\fh_g$ for the Siegel space of symmetric $g \times g$ complex matrices whose imaginary parts are
positive definite.  For a fixed sequence $D = (d_1,\ldots,d_g)$ of positive integers with $d_1 \vert d_2 \vert \cdots \vert d_g$,
we obtain the notion of a polarization of type $D$.  For $\tau \in \fh_g$, we associate the complex torus
$X^D_\tau(\CC) := \CC^g / (\tau \ZZ^g + \operatorname{diag}(d_1,\ldots,d_g) \ZZ^g)$ which, because of the choice of polarization,
is actually an abelian variety.  There is a discrete subgroup $G_D$ of the symplectic group $\operatorname{Sp}_{2g}(\QQ)$ for which
the quotient $G_D \backslash \fh_g$ has the structure of an algebraic variety $\cA_{g}^D (\CC)$ (in this case, ``$\cA$'' does not refer
to an arc space).  As in the case of $g = 1$, the quotient map is ramified, but in passing to congruence subgroups one obtains
covering maps.  Theorem 1.1 of~\cite{PS-theta} asserts that
is a fundamental domain on which the covering map $\fh_g \to \cA_{g,D}(\CC)$ is definable.

The covering map for the abelian variety associated to $\tau \in \fh_g$ is itself uniformly defined.
That is, if we were to take $U := \fh_g \times \CC^g$, then to a pair $(\tau,z)$ we may associate the abelian variety $X^D_\tau$ and the
point $z \mod (\tau \ZZ^g + \operatorname{diag}(d_1,\ldots,d_g) \ZZ^g) \in X^D_\tau(\CC)$.  The quotient
$(G_D \ltimes \ZZ^{2g}) \backslash (\fh_g \times \CC^g)$ has the structure of an algebraic variety $\cX_{g}^D$ and fits into the following
commutative diagram

$$
\xymatrix{
\fh_g \times \CC^g \ar[d] \ar[rrr] & & & \fh_g \ar[d] \\
\cX^D_g(\CC) \ar[rrr] & & & \cA^D_g(\CC)
}
$$

Theorem 1.3 of~\cite{PS-theta} asserts that the covering map of $\cX^D_g$ is definable on a fundamental domain.  Passing to a finite cover of $\cA^D_g$ corresponding
to a suitable congruence subgroup (or, equivalently, to the choice of some level structure when interpreted in terms of the moduli problem),
the resulting map is a covering map to which Theorem~\ref{mainthm} applies directly.

\subsubsection{Arithmetic varieties}
Klingler, Ullmo and Yafaev have shown~\cite{KUY} that for \emph{any} arithmetic variety, the associated covering map is definable in $\RR_{\text{an},\exp}$ on
some fundamental domain.  Here, an arithmetic variety is a complex algebraic variety expressible as $\Gamma \backslash D$ where $D$ is a symmetric
Hermitian domain and $\Gamma$ is an arithmetic group.   This definability theorem supersedes the Peterzil-Starchenko definability result
for the covering maps for the moduli spaces of abelian varieties and includes the covering maps for all Shimura varieties.
However, the requirement that $D$ be a symmetric Hermitian domain precludes its application to the case of the universal
abelian variety or more generally to mixed Shimura varieties.   Gao has shown~\cite{Gao} that the covering maps associated to all
mixed Shimura varieties are $\RR_{\text{an},\exp}$-definable when restricted to some fundamental domain.  Thus,
from Theorem~\ref{mainthm} and the Klingler-Ullmo-Yafaev and Gao definability theorems,
we deduce that the generalized logarithmic derivatives associated to  arithmetic
varieties and to mixed Shimura varieties are differential constructible.

\subsection{Picard-Fuchs equations, periods and Manin homomorphisms}

Families of complex algebraic varieties often come equipped with period
mappings.  As we cannot say as much as we would like in full generality, we shall
confine ourselves to a na\"{\i}ve discussion referring the reader to~\cite{Griffiths}
for more details.

If $X$ is a complex algebraic variety and $\omega_1, \ldots, \omega_g$ is a basis of
global one-forms on $X$ and $\gamma_1, \ldots, \gamma_\ell$
is a basis of the free part of the integral homology group $H_1(X,\ZZ)$, then the matrix
$(\int_{\gamma_i} \omega_j) \in M_{\ell \times g} (\CC)$ is called a period matrix of $X$.  Of course,
changing bases would result in a different matrix so that the period matrix of $X$ is well defined only
up to the action of $\operatorname{GL}_g(\CC)$ on one side and $\operatorname{GL}_\ell(\ZZ)$ on the other.
It is a fact that if $X$ varies in an algebraic family, $X \to S$, and the forms $\omega_1, \ldots,\omega_g$ are also
taken to vary algebraically, then one may choose the homology classes $\gamma_i$ to vary semi-algebraically so that
locally on $S(\CC)$ the components of the period matrix are analytic functions.  The theory of Picard-Fuchs equations
shows that period matrices satisfies linear differential equations with coefficients from the function field of $S$.

We explain now how to deduce the existence of the Picard-Fuchs equations with coefficients in $\CC(S)$ from
the global $\RR_{\text{an},\exp}$-definability of some branch of functions extending
extending $s \mapsto \int_{\gamma_i} (\omega_j)_s$.  In fact, we shall show a little more: not only do the
period matrices satisfy linear differential equations over $\CC(S)$, but one may differentially constructibly
find linear differential operators so that for every differential field $M$ of meromorphic
functions extending $\CC(S)$ and every point $a \in S(M)$ the kernel of the associated linear
differential operator is exactly the $\CC$-vector space generated by the periods of $X_s$.

Let us note that it follows from the Peterzil-Starchenko theorem on the
definability of the covering maps for universal families of abelian varieties that for every algebraic family of abelian
varieties, global branches of the period functions are definable.  However, we do not know whether such an hypothesis on
definability holds for all families of varieties.

With the following definition we make precise what we mean by a local system and what it means for
such a local system to definably trivialize.

\begin{Def}
Let $S$ be a complex algebraic variety and $N$ and $g$ be two natural numbers. 
A \emph{$\ZZ^g$-local system} over $S$ is a subset $\Gamma \subseteq {\mathbb G}_{a,S}^N(\CC)
= S(\CC) \times \CC^N$ such that 
\begin{itemize}
\item for each $s \in S(\CC)$, $\Gamma_s \subseteq \CC^N$ is a subgroup isomorphic to $\ZZ^g$ and
\item $S(\CC)$ is covered by open sets $U \subseteq S(\CC)$ on which there are analytic functions
$\gamma_{1,U}, \ldots, \gamma_{g,U}:U \to \CC^N$ so that for all $s \in S(\CC)$ one has
$\Gamma_s = \sum_{i=1}^g \ZZ \gamma_{i,U}(s)$.
\end{itemize}

Fix an o-minimal structure $\RR_{\cF}$ expanding the real field.  In practice, 
one may take $\RR_{\exp,\text{an}}$. 
We say that $\Gamma$ \emph{definably trivializes} 
if there are definable (though not necessarily continuous) functions $\nu_1, \ldots, \nu_g:S(\CC) \to \CC^N$ so that
at each point $s \in S(\CC)$ one has $\Gamma_s = \sum_{i=1}^g \ZZ \nu_i(s)$.  
\end{Def}

We begin with an algebraic version of our theorem on Picard-Fuchs equations similar to Proposition~\ref{mainthmalg}.

\begin{prop}
\label{PFalg}
Let $S$ be a complex algebraic variety and 
$\Gamma \subseteq {\mathbb G}_{a,S}^N(\CC) = S(\CC) \times \CC^N$ a $\ZZ^g$-local system
over $S$ by which definably trivializes via definable functions 
$\nu_1, \ldots, \nu_g:S(\CC) \to \CC^N$.  
Then there is a
constructible function $\Phi:S \to \operatorname{Mat}_{N \times N}$ so that for each
$s \in S(\CC)$ the $\CC$-vector subspace of $\CC^N$ generated by $\Gamma_s$ is the
kernel of $\Phi(s)$.
\end{prop}

\begin{proof}
For each $j \leq N$, we denote by $G(j,N)$ the Grassmannian of $j$-dimensional linear subspaces
of ${\mathbb G}_a^N$.   We shall prove this proposition by first showing that the association
$\sigma:S \to \bigcup_{j=1}^g G(j,N)$ given by sending $s \in S(\CC)$ to the $\CC$-vector space
generated by $\Gamma_s$ is constructible.  Then, we shall observe that there is a constructible
function $\Psi:G(j,N) \to \operatorname{Mat}_{N \times N}$ so that for each $[V] \in G(j,N)(\CC)$
the kernel of $\Phi([V])$ is precisely $V$.  Our map $\Phi$ is then $\Psi \circ \sigma$.

For $1 \leq j \leq g$, let $S_j(\CC) := \{ s \in S(\CC) ~:~ \dim_\CC \CC \cdot \Gamma_s \leq j \}$.
Working in an open set $U \subseteq S(\CC)$ on which $\Gamma$ trivializes, we see that
$S_j(\CC) \cap U$ is an analytic set as it is defined by the vanishing of a collection of determinants
of certain minors of the matrix $(\gamma_{1,U}(s),\ldots,\gamma_{g,U}(s))$.   Hence, $S_j$ is analytic.
On the other hand, using the natural definition of dimension and the functions $\nu_1, \ldots, \nu_g$,
one sees that $S_j$ is definable as well.  Hence, by Theorem~\ref{PSGAGA}, $S_j$ is an algebraic
subvariety of $S$.

Likewise, for each $j$, the function $\sigma_j:(S_j \smallsetminus S_{j-1}) \to G(j,N)$
given by sending $s \in (S_j \smallsetminus S_{j-1})(\CC)$ to the code for the vector space
generated by $\Gamma_s$ is both analytic and definable, and, hence, a regular map.   The
function $\sigma:S(\CC) \to \bigcup_{j=1}^g G(j,N)(\CC)$ is simply the union of the various $\sigma_j$.

Let us prove by noetherian induction that if $Z \subseteq G(j,N)$ is an irreducible subvariety, then there is a
constructible function $\Psi_Z:Z \to \operatorname{Mat}_{N \times N}$ so that if $[V] \in Z$ is the
code for the vector space $V \leq {\mathbb G}_a^N$, then $V = \ker \Psi_Z([V])$.
Taking $Z = G(j,N)$, we obtain our desrired $\Psi$.
Let $[V] \in Z(\CC(Z))$ be the generic point.   By basic linear algebra, there is some
linear map $L:\CC(Z)^N \to \CC(Z)^N$ for which $\ker L = V(\CC(\ZZ))$.  Localizing, we
find $Z' \subsetneq Z$ a proper subvariety of $Z$ for which the coefficients of
$L$ are regular functions on $Z \smallsetminus Z'$.  By induction, there is a constructible function
$\Psi_{Z'}:Z' \to \operatorname{Mat}_{N \times N}$ with the desired properties.  Let $\Psi_Z$ be
given by $\Psi_{Z'}$ on $Z'$ and by associating $z \in Z \smallsetminus Z'$ the specialization of
$L$ to $z$ on $Z \smallsetminus Z'$.
\end{proof}

As with our proof of Theorem~\ref{mainthm}, we find our tighter Picard-Fuchs operators from
the ostensibly weaker algebraic Proposition~\ref{PFalg}.

\begin{theorem}
\label{PFdiff}
Let $S$ be a complex algebraic variety.   Give $\CC(S)$ the structure of a differential field by
fixing a basis of $\CC$-derivations.  Let $N$ and $g$ be two natural numbers. Suppose
that $\Gamma \subseteq {\mathbb G}_{a,S}^N(\CC) = S(\CC) \times \CC^N$ is a $\ZZ^g$-local system
over $S$ with local trivializations as in Proposition~\ref{PFalg}. Also as in Proposition~\ref{PFalg}, we
suppose moreover that $\Gamma$ definably trivializes via functions $\nu_1, \ldots, \nu_g:S(\CC) \to \CC^N$.
Then there is a differential constructible function $\Theta$ on $S$ taking values in a space of linear
differential operators so that for any differential field $M$ of meromorphic functions extending $\CC(S)$ and
point  $s \in S(M)$ the $\CC$-vector subspace of $\CC^N$ generated by $\Gamma_s$ is the
kernel of $\Theta(s)$.
\end{theorem}
\begin{proof}
Working by noetherian induction on the Kolchin topology on $S$, one sees that there is some
number $\ell \gg 0$ so that for any differential field $M$ of meromorphic function extending $\CC(S)$
and point $s \in S(M)$ the $\CC$-vector space generated by $\Gamma_s$ is the pre-image under
$\nabla_\ell$ of the $M$-vector space generated by $\nabla_\ell(\Gamma_s)$ in $\cA_\ell ({\mathbb G}_a^N)$.

Let us define $\Gamma^{(\ell)}$.  For $U \subseteq S(\CC)$ an open set over which $\Gamma$
trivializes and $t \in J_\ell(U) \subseteq \cA_\ell S(\CC)$, let $\Gamma_t^{(\ell)} := \sum_{i=1}^g \ZZ J_\ell(\gamma_{i,U})(t)$.
From the local patching for $\Gamma$, it is easy to see that $\Gamma^{(\ell)}$ is a $\ZZ^g$-local system over
$\cA_\ell S$.

Replacing $S$ by $\cA_\ell S$, $\Gamma$ by $\Gamma^{(\ell)}$, and $N$ by $N' := \dim \cA_\ell ({\mathbb G}_a^N)$,
Proposition~\ref{PFalg} applies giving a constructible function $\Phi: \cA_\ell({\mathbb G}_a^N) \to
{\mathbb G}_a^{N'}$ so that for any $t \in \cA_\ell S(M)$, the kernel of $\Phi(t)$ is the
$M$-vector space generated by $\Gamma_t^{(\ell)}$.   Define $\Theta$ on $S$ by $\Theta(s) := \Phi(\nabla(s)) \circ \nabla$.
Then for $s \in S(M)$,  $\ker \Theta(s)$ is the $\CC$-vector space generated by $\Gamma_s$.
\end{proof}

\begin{Rk}
If one drops either the hypothesis that $\Gamma$ is a $\ZZ^g$-local system (in the sense that 
it locally analytically trivializes) or the hypothesis that $\Gamma$ definably discontinuously 
trivializes, then the conclusion of Theorem~\ref{PFdiff} fails.  For example, if $E \to S$ is an 
elliptic scheme over an algebraic curve $S$, then there is a definable function 
$\rho:S(\CC) \to \CC$ so that for each $s \in S(\CC)$, $\rho(s)$ is a nonzero element of the
kernel of the exponential map $\exp_{E_s}:\CC \to E_s(\CC)$.  It is well-known that if 
$E \to S$ is not isotrivial, then $\rho$ cannot satisfy an order one algebraic 
differential equation (even where it is analytic). 
Theorem~\ref{PFdiff} does not apply as $\Gamma := \ZZ \cdot \rho \subseteq S(\CC) \times \CC$ 
is not a local system since it does not locally analytically trivialize near the discontinuities of 
$\rho$.   On the other hand, the Bessel function does give rise to a $\ZZ$-local system, but it
too satisfies an order two algebraic differential equation, but not one of order one.  The 
issue is that the local system does not definably trivialize.  

I thank D. Bertrand for suggesting these examples.  In the same commuication he suggests that in line 
with the classical Picard-Fuchs theory the hypotheses of Theorem~\ref{PFdiff} should correspond to the 
assertions that the associated differential equations are invariant under monodromy and have at worst
regular singularities.  
\end{Rk}

Let us apply Theorem~\ref{PFdiff} to the case of periods of abelian varieties.

\begin{cor}
\label{PFav}
Let $S$ be an irreducible complex algebraic variety and $X \to S$ an abelian scheme over $S$ of relative dimension $g$.
We regard $\CC(S)$ as a differential field by a fixing a basis $\partial_1, \ldots, \partial_n$ of commuting derivations on
$\CC(S)$.
Then there is a linear map $\widetilde{\Phi}:(\cA_{2g} {\mathbb G}_{a}^g)_S \to ({\mathbb G}_a^M)_S$ over $S$ so that
the kernel of $\Phi := \widetilde{\Phi} \circ \nabla_{2g}$ is the $\CC$-vector bundle generated by $\Gamma := \{ (s,z) \in S(\CC) \times
\CC^g ~:~ z \in \ker \exp_{X_s} \}$.
\end{cor}
\begin{proof}
By the main theorem of~\cite{PS-theta}, when $X \to S$ is a universal family of abelian varieties, there is a definable (discontinuous)
splitting of $\Gamma$.  For a general family of abelian varieties, pull back the definable splitting from the universal family.
\end{proof}

Corollary~\ref{PFav} justifies Manin's analytic construction of differential additive characters on abelian varieties~\cite{Manin}.

\begin{cor}
Let $S$ be an irreducible complex algebraic variety and $X \to S$ an abelian scheme over $S$ of relative dimension
$g$.  Then there is a map of
differential algebraic groups over $S$, $\mu:X \to ({\mathbb G}_a^N)_S$ (for some $N$)  so that for any differentially
closed field $(\UU,\partial_1, \dots, \partial_n) \supseteq (\CC(S),\partial_1,\ldots,\partial_n)$ and any
point $s \in S(\UU)$, the differential function field of the kernel of $\mu_s:X_s(\UU) \to {\mathbb G}_a^N(\UU)$
has transcendence degree at most $2g$.
\end{cor}
\begin{proof}
The exponential map $\exp_X:S(\CC) \times \CC^g \to X(\CC)$ is an analytic covering map over $S$.  As such, it has a local analytic
inverse $\log_X:X(\CC) \to S(\CC) \times \CC^g$.  We shall abuse notation somewhat by regarding $\log_X$ as a map to $\CC^g$.
The many valued function $\log_X$ is well defined up to addition by $\Gamma$.  Hence, if $\Phi$ is the linear
 differential map of Corollary~\ref{PFav}, then $\mu := \Phi \circ \log_X$
  is a well-defined mapping.  Arguing as in the proof of Theorem~\ref{mainthm},
 namely by expressing $\Phi \circ \log_X$ as a composite of regular maps on jet spaces with $\nabla_{2g}$,
 we see that $\mu$ is differential algebraic and from its construction because the fibres of $\Phi$ are vector spaces
 over the constants of dimension at most $2g$, it follows that the fibres of $\mu$ have dimension at most $2g$.
\end{proof}

\section{Concluding remarks and questions}
\label{cq}

Let us conclude with some natural questions raised by our construction of generalized
logarithmic derivations.

\subsection{Differential equations for covering maps}

Let $\pi:U \to X(\CC)$ be a covering map as in the statement of Theorem~\ref{mainthm}.
Let $M = \cM(U)$ be the field of meromorphic functions on $U$ treated as differential field
with respect to some choice of a basis of derivations $\partial_1,\ldots,\partial_n$.
Let $t \in U(M)$ be the $M$-rational point corresponding to the identity map $U \to U$.
Then $\pi(t) \in X(M)$ ``is'' $\pi$ regarded as an $M$-rational point of $X$.  From Theorem~\ref{mainthm}
we know that $\pi(t)$ satisfies the algebraic differential equation $\chi(\pi(t)) = \widetilde{\chi}(t)$.
Noting that $t$ is the restriction of the generic point of $Y$ to $U$ and that the derivations
on $M$ may be chosen to be the restriction of a basis of derivations of $\CC(Y)$, we
see that the differential field generated by $\pi(t)$ over $\CC(Y)$  has transcendence degree at most
$\dim G$ over $\CC(Y)$.  Since $\CC(Y)$ is a differential field of finite transcendence degree over $\CC$, it follows
that the differential field generated by $\pi(t)$ over $\CC$ has finite transcendence degree.

However, in some cases, for example for the analytic $j$-function $j:\fh \to \AA^1(\CC)$ and for exponential functions
$\exp_G:\CC^g \to G(G)$ of connected, commutative algebraic groups, with the usual differential structure,
the differential field generated by $\pi(t)$ over
$\CC$ has transcendence degree exactly $\dim(G)$ over $\CC$.  That is, $\widetilde{\chi}(t)$ is a constant point.

\begin{question}
Is it always the case that there is some choice of a basis of $\operatorname{Der}(\CC(Y)) = \operatorname{Hom}_{\CC(Y)}(\Omega_{\CC(Y)/\CC},\CC(Y))$
so that $\widetilde{\chi}(t)$ is a constant point?   If not, under what conditions is this the case?
\end{question}

\subsection{Classification theory and the fibres of $\chi$}

In~\cite{FS}, Freitag and the present author show that the differential equation satisfied by the $j$-function is strongly minimal
and has trivial forking geometry.  It is not unreasonable to expect that other generalized logarithmic derivatives will produce examples
of types with trivial forking geometry but complicated binary structure coming from generalized Hecke correspondences.  We do not
propose a precise statement of a conjecture on the classification theoretic structure of fibres of generalized logarithmic derivatives
as that would involve delineating the ways in which sets nonorthogonal to the constants and to Manin kernels may arise.   

\begin{problem}
Describe the classification theoretic structure (\emph{i.e.}, an analysis in minimal types placing these types into their positions
relative to the Zilber trichotomy) of fibres of generalized logarithmic derivatives in terms of group theoretic data of the covering
spaces.
\end{problem}

\subsection{Ax-Schanuel problems}
In~\cite{Ax}, Ax proved a function field version of Schanuel's conjecture on algebraic relations amongst exponentials as a
purely differential algebraic theorem.  In~\cite{PiTs}, Pila and Tsimmerman prove a variant of Ax's theorem for the $j$-function.
They observe that via the Seidenberg embedding theorem, their Schanuel-like theorem for the $j$-function admits a differential
algebraic formulation, though their proof involves other techniques.   Our differential equations permit the differential algebraic formulation
of general transcendence conjectures for covering maps.  As we have not resolved the question of which ``obvious'' relations might hold, we do
not suggest a precise general conjecture here.

\bibliographystyle{siam}
\bibliography{ade}
\end{document}